\documentclass[a4paper,11pt]{amsart}

\usepackage{a4wide}
\usepackage{centernot}
\usepackage{url}

\usepackage{color}
\usepackage{tensor}
\usepackage{enumitem} 
\usepackage{amsfonts,amssymb,amsmath,amsthm,graphicx,latexsym}
\usepackage{mathrsfs,dsfont}

\newcommand{\clb}{\clb}

\newcommand{\IE}{\mathbb{E}}
\newcommand{\innerl}{\Big\langle}
\newcommand{\innerr}{\Big\rangle}
\newcommand{\IP}{\mathbb{P}}
\newcommand*{\Comb}[2]{{}^{#1}C_{#2}}%

\usepackage{dsfont}

\theoremstyle{definition}
\newtheorem{dfn}{Definition}[section]

\theoremstyle{plain}
\newtheorem{thm}[dfn]{Theorem}
\newtheorem{lmma}[dfn]{Lemma}

\newtheorem{ppsn}[dfn]{Proposition}

\theoremstyle{remark}

\newtheorem{rmrk}[dfn]{Remark}

\newcommand{\bdfn}{\begin{dfn}}
\newcommand{\bthm}{\begin{thm}}

\newcommand{\IR}{\mathbb{R}}

\newcommand{\cla}{\mathcal{A}}

\newcommand{\ot}{\otimes}

\usepackage{comment}

\newtheorem{theoremA}{Theorem}

\newcommand{\tp}{\mathbin{\hbox{$\bigcirc$\hbox to
      0pt{\hspace{-0.81em}$\scriptstyle\top$\hfil}}}}

\begin{document}
\title[
  Quadratic number of nodes suffices to learn a dataset.
]%
{
  Quadratic number of nodes is sufficient to learn a dataset via gradient descent.
}
\author{Biswarup Das}
\address{} 
\address{Biswarup Das}
\address{Laboratory of advanced combinatorics and network applications.}
\address{Moscow Institute of Physics and Technology, Russia.}
\email{das.b@mipt.ru} 
\author{Eugene A. Golikov}
\address{}
\address{Eugene A. Golikov}
\address{Laboratory of neural systems and deep learning.}
\address{Moscow Institute of Physics and Technology, Russia.}
\email{golikov.ea@mipt.ru}

\begin{abstract}
We prove that if an activation function satisfies some mild conditions and number of neurons in a two-layered fully connected neural network with this activation function is beyond a certain threshold, then gradient descent on quadratic loss function finds the optimal weights of input layer for global minima in linear time. 
This threshold value is an improvement over previously obtained values in \cite{gdescentOverPara1, gdescentOverPara2}. 
\end{abstract}

\maketitle 

\section{Introduction}

During the last decade deep learning has achieved remarkable success in several fields of practical applications, despite the fact that the reason for its success is still largely unclear.
In particular, it is not yet well understood why neural nets used in practice generalize to data points not used in its training procedure.
Alongside with this unexplained phenomenon, there is another issue, which is relatively simpler to state: {\it why do neural nets learn the training data?}
In other words, it is not clear why our training algorithm (say gradient descent or stochastic gradient descent) manages to find a solution with zero training loss.

There are theoretical results based on approximation theory~\cite{proofNeural} which establish that a large enough neural network with suitable weights can approximate any function.
However, such theoretical considerations do not touch upon the issues raised above, for example, it is not  guaranteed that the corresponding optimal weights associated with the network can be found by performing say, a stochastic gradient descent (SGD) (one of the most commonly used optimization procedure).

In practice however, we see that for networks of practical sizes SGD always succeeds in finding a point of the global minima of the associated loss function.
One step towards understanding this phenomenon at a theoretical level was made by Du et al.~\cite{gdescentOverPara1}, who proved that with a probability of $1-\delta$, gradient descent finds global minimum of the loss function corresponding to an $l_2$ regression problem solved by a two-layer fully-connected network with ReLU activation, given that the number of hidden neurons is at least $\Omega \left(\frac{n^6}{\lambda_0^4\delta^3}\right)$, where $n$ is the number of training examples and $\lambda_0$ is the smallest eigenvalue of the gramian matrix made with the training data. Recently, Song \& Yang~\cite{gdescentOverPara2} improved this result to $\Omega \left(\frac{n^4}{\lambda_0^4}\ln\left(\frac{n}{\delta}\right)\right)$. 

In a similar vein, our current work further improves this result to $\Omega \left(\frac{n^2}{\lambda_0^2}\ln\left(\frac{n}{\delta}\right)\right)$ under some mild assumptions on activation function that are satisfied by many ones used in practice. We hypothesise that this result is no longer improvable using the same technique.

\subsection*{Acknowledgement}

This work was supported by National Technology Initiative and PAO Sberbank project ID 0000000007417F630002.

\section{A new activation function and statement of the main result}

\subsection{A necessary and sufficient condition for a function to be an activation function}

In \cite[Theorem 1]{proofNeural} the author gives a necessary and sufficient condition for a function $\sigma:\IR\longrightarrow\IR$ to be an activation function for a neural network. We state the result below:
\begin{ppsn}\label{Proposition: From the neural proof paper}{\cite[Theorem 1]{proofNeural}}
Suppose $\sigma:\IR\longrightarrow\IR$ is a function which is in $L^\infty_{\text{loc}}(\IR)$ and the closure of the set of discontinuities of $\sigma$ are of Lebesgue measure zero. Define 
\[
\Sigma_n:=\text{span}\{\sigma(w^Tx+b):~w\in\IR^n,b\in\IR\}.  
\]
Then $\Sigma_n$ is dense in $C(\IR^n)$ in the compact open topology if and only if $\sigma$ is not an algebraic polynomial (a.e.).
\end{ppsn}

It is worthwhile to note that all the standard activation functions used in neural networks satisfy the hypothesis of the above proposition. We will prove that under some mild conditions on an activation function, a two-layered neural network with number of hidden neurons beyond a certain improved threshold, can always be trained by gradient descent in linear time. Our new improved threshold is an improvement over previous threshold values obtained in \cite{gdescentOverPara1, gdescentOverPara2}.

\subsection{Our set-up}\label{Subsection: our setup}
Suppose $\sigma:\IR\longrightarrow\IR$ is a function which satisfies the conditions of Proposition \ref{Proposition: From the neural proof paper}, so that the single hidden layered fully connected neural network with this activation function given by 
\begin{equation}\label{Equation: Neural network}
f(W,x,a,b):=\frac{1}{\sqrt{m}}\sum_{r=1}^ma_r\sigma(w_r^Tx+b_r), 
\end{equation}
where $W:=\{w_i\}_{i=1}^m$ are the input weights, $b:=\{b_i\}_{i=1}^m$ are the input biases and $a:=\{a_i\}_{i=1}^m$ are the output weights, can approximate any element in $C(\IR^n)$ in the compact open topology with the right choice of $m, a_r,w_r,b_r$. To facilitate easy computation and avoid unecessary complications with notations, we will consider the above mentioned neural network without the ``bias" term, i.e. our working model will be 
\[
f(W,x,a):=\frac{1}{\sqrt{m}}\sum_{r=1}^ma_r\sigma(w_r^Tx). 
\]

\subsection*{Assumptions on the activation function}
The only assumptions we make on the activation function $\sigma$ are the following:
\begin{enumerate} 
\item[(a)]
{\it We assume that there exist constants $c_1,c_2$ such that $|\sigma^\prime(x)|<c_1$ and $|\sigma^{\prime\prime}(x)|<c_2$ for all $x\in\IR$.}
\item[(b)]
{\it Let $w\sim N(0,\mathbb{I}_d)$. We then assume that $\Big|\IE\Big[\sigma(w^Tx)^2\Big]\Big|\leq c_3$ for some constant $c_3$.}
\end{enumerate}
An important example of such an activation function is given  by the softplus activation function defined as
$\sigma(x):= \ln(1+e^{x})$.
Clearly, it satisfies Assumption (a). Using the inequality $\frac{x}{1+x}\leq \ln(1+x)\leq x$ for all $x>0$, it can be shown that $\sigma$ also satisfies Assumption (b). 

\subsection*{}
Suppose we are given a training data $\{x_i\}_{i=1}^n$, where for each $i$ $x_i\in\IR^d$ and corresponding responses $\{y_i\}_{i=1}^n$, where $y_i\in\IR$ for each $i$. Denoting the collection of input weights of the neural network \eqref{Equation: Neural network} by $W$ and output weights as $a$, we define the quadratic loss function $L(W,a)$ as
\begin{equation}\label{Equation: Quadratic loss function}
L(W,a):=\frac{1}{2}\sum_{i=1}^n(y_i-f(W,x_i,a))^2.
\end{equation}
We will give a new lower bound on $m$, the number of neurons in the hidden layer, and will prove that performing continuous-time gradient descent on the loss function $L$ with that many neurons, leads us to zero-loss solution, with the convergence rate being exponential. This new lower bound is better than the bounds available so far in the literature (\cite{gdescentOverPara1, gdescentOverPara2}). More precisely, we prove the following result:

\renewcommand{\thetheoremA}{\ref{Theorem: main result}}
\begin{theoremA}
Consider an activation function $\sigma(x)$ satisfying the assumptions mentioned in Subsection \ref{Subsection: our setup}. 
Suppose we are given a training data $\{x_i\}_{i=1}^n$ with $x_i\in\IR^d$, $\|x_i\|\leq1$, such that $x_i \neq x_j \; \forall i,j$, and responses $\{y_i\}_{i=1}^n$ with $y_i\in\IR$ and $|y_i|<\kappa$ for some number $\kappa$. Let $\lambda_0$ be the minimum eigenvalue of the matrix $H^\infty$, whose entries are given by 
\[
H^\infty_{pq}:=\IE_{z\sim N(0,\mathbb{I}_d)}\Big[\sigma^\prime(z^Tx_p)\sigma^\prime(z^Tx_q)\Big]\ot x_p^Tx_q\quad(p,q=1,2,\cdots n),.
\]
Then we have the following:

\begin{itemize}

\item[(a)]
$\lambda_0>0$.

\item[(b)]
Fix a $\delta>0$ such that $n\delta+\frac{D}{4c_1c_2\ln(\frac{2n}{\delta})}<1$, where $D:=\sqrt{\kappa^2+c_3}$. 
Let us select $m>\frac{64c_1^2c_2^2n^2\ln(\frac{2n}{\delta})}{\lambda_0^2}$ ($c_1,c_2$ are the constants appearing in the assumption in Subsection \ref{Subsection: our setup}) and consider the network with a single hidden layer:\[f(W,x,a):=\frac{1}{\sqrt{m}}\sum_{r=1}^m a_r\sigma(w_r^Tx).\]
Then with random initializations: $a_r\sim\text{unif}\{-1,1\}$ and $w_r(0)\sim N(0,\mathbb{I}_d)$ for all $r=1,2,\cdots m$ and the condition that we do not train the output layer i.e. $a_r$'s are kept fixed upon initialization, with probability at least $1-\delta n-\frac{D}{4c_1c_2\ln(\frac{2n}{\delta})}$, gradient descent with small enough step size converges to $y_i$'s at an exponential rate.

\item[(c)]
Define $\delta' := \delta n+\frac{D}{4c_1c_2\ln(\frac{2n}{\delta})}$.
There exists a constant $C$ such that $m > \frac{C n^2\ln(\frac{2n}{\delta'})}{\lambda_0^2}$ implies $m>\frac{64c_1^2c_2^2n^2\ln(\frac{2n}{\delta})}{\lambda_0^2}$.
Then, as a corollary of point (b) we have that with random initialization as in the previous point, with probability at least $1-\delta'$, gradient descent with small enough step size converges to $y_i$'s at an exponential rate.

\end{itemize}
\end{theoremA}

\subsection{Why this new bound on $m$ is better than the available bounds}

To the best of our knowledge, the first mathematical proof that with high probability, gradient descent indeed finds the optimal input weights while training a two-layered neural network with ReLU activation function, provided the number of neurons in the hidden layer is above a certain threshold, appeared in the work \cite{gdescentOverPara1}. 
The threshold was proven to be $\Omega\left(\frac{n^6}{\lambda_0^4\delta^3}\right)$ where $\delta$ is the failure probability (see \cite[Theorem 3.2]{gdescentOverPara1}). 
Mathematically this result is sound however, it is too loose: indeed, as we observe in practice gradient descent easily finds a zero-loss solution with much fewer number of hidden nodes \cite[Section 6]{DLM}, a phenomenon whose mathematical explanation is still lacking to the best of our knowledge.

A subsequent improvement of the lower bounds appeared in the work \cite[Theorem 1.4]{gdescentOverPara2} namely, the bound was brought down to $\Omega\left(\frac{n^4\ln(\frac{n}{\delta})}{\lambda_0^4}\right)$, which is still too loose.
Although in \cite[Theorem 1.6]{gdescentOverPara2}, a threshold which looks like $\Omega\left(\frac{n^2\alpha(\alpha+\theta^2)\ln(\frac{n}{\delta})}{\lambda_0^4}\right)$ has been proposed, this threshold introduces two quantities: $\alpha$ and $\theta$, which depend on training dataset.
In the worst case $\alpha = n$ and $\theta = \sqrt{n}$.
Hence \cite[Theorem 1.6]{gdescentOverPara2} does not improve over quartic lower bound of \cite[Theorem 1.4]{gdescentOverPara2} in the general case.
In contrast, our result holds {\it without any additional assumptions on training data}.

A recent work \cite{kawaGradientDescent} has considered the problem of convergence of gradient descent in light of deep neural networks with analytic activation functions. 
It is worthwhile to say a few words about this work in the present context. 
In particular, \cite[Theorem 1]{kawaGradientDescent} applied to a single layered fully connected network apparently yields a threshold $\Omega(n)$. 
However, we would like to point out that the consideration in \cite{kawaGradientDescent} is fundamentally different from ours as well as that in \cite{gdescentOverPara1, gdescentOverPara2}, within the context of a network with single hidden layer.
A closer look at the proof of \cite[Theorem 1]{kawaGradientDescent} after equation (19) in \cite[pp. 5]{kawaGradientDescent} reveals that, the proof proceeds by enforcing zero learning rate for all but the last (output) layer. 
To put it differently, if we train our single layered network given by Equation \eqref{Equation: Neural network} by the methodology proposed in the proof of \cite[Theorem 1]{kawaGradientDescent} upon making some random initialization, the only weights which are being updated in the sequel are the output weights namely $a_1,a_2,\cdots a_m$. 
This means that we are essentially training a single layered linear network with gradient descent on randomly extracted features.

Let us show that in this case if the given number of nodes is at least $n$, and all of the training points are different, gradient descent finds the global minima almost surely with respect to initialization.
We give a brief proof of this. 
Recall the quadratic loss from Equation \eqref{Equation: Quadratic loss function}. 
Denoting it by $L$, it follows that gradient of $L$ with respect to the output weights only has components given by 
\[
\frac{\partial L}{\partial a_r}=\sum_{i=1}^n\Big(y_i-f(W,x_i,a,b)\Big)\sigma(w_r^Tx_i+b_r)\quad(r=1,2,\cdots m).
\]
Denoting the gradient by $\nabla_aL$, we see that $\nabla_aL=Ae$, where $A\in\IR^{m\times n}$ is the matrix given by $A_{pq}:=\sigma(w_p^Tx_q+b_p)$ for $p=1,2,\cdots m$ and $q=1,2,\cdots n$ and $e\in \IR^n$ is the vector whose components are $e_i:=y_i-f(W,x_i,a,b)$. 
At a critical point we must have that $\nabla_aL=0$ which means $Ae=0$. Multiplying the last equation from the left by $A^T$, this means $A^TAe=0$. 
If $A^TA$ is invertible, this would mean that the only critical point is the one for which $e=0$ i.e. $f(W,x_i,a,b)=y_i$ for all $i=1,2,\cdots n$, which in that case would be a point of global minima. 
$A^TA$ will be invertible if and only if $\det (A^TA)\neq 0$. Now we may observe that $\det (A^TA)$ is an analytic function of the input weights and biases. 
By \cite[Lemma 1.2]{NguyenComplexPowers} it follows that either the set of zeroes of the analytic function $\det (A^TA)$ are of Lebesgue measure zero, or $\det (A^T A)$ is identically zero.
Since the number of nodes $m$ is not less than the number of training points $n$, and all of the training points are different, there exist such $A$ for which $\det (A^T A)$ is not zero.
Hence the first case holds.
This means if we randomly initialize the input weights and biases by sampling from any distribution which has a density with respect to the Lebesgue measure, almost surely we will have the matrix $A^TA$ invertible, so that the loss function will have a unique critical point, which would be a point of global minima. 
It follows now that gradient descent in this case will always converge to a global minimum with zero loss.
In practice, however, networks are usually trained as a whole.
Training the last layer only corresponds to so-called ``kernel" or ``lazy training" regime~\cite{lazy_training}, which is not our consideration.
The above-mentioned work argues that it is unlikely that such a lazy training regime is behind many succeses of deep learning.

\subsection*{}


In the next section we collect some preliminaries, towards proving our main result, Theorem \ref{Theorem: main result} in Subsection \ref{Subsection: Convergence of the gradient descent}.
\section{Preliminaries}
Throughout the section and henceforth, $\ot$ will denote the tensor product of matrices and $\oplus$ will denote the direct sum of matrices. We will not define these notions, as they could be found in any standard text book on linear algebra.
\subsection{Khatri-Rao product of matrices}\label{Subsection: Khatri-Rao product of matrices}
Suppose $A\in\IR^{m\times r}$ and $B\in\IR^{n\times r}$ be matrices. The  Khatri-Rao product of $A$ and $B$ \cite[Lemma 13]{khatrirao} denoted $A\odot B$ is a $mn\times r$ matrix, whose $i^{\text{th}}$ column is obtained by taking tensor product of the $i^{\text{th}}$ columns of $A$ and $B$. For example, let $A=\begin{pmatrix}
                                                           a&b&c\\d&e&f\\x&y&z                                                                                                                                                                                                                                                      
                                                                                                                                                                                                                                                                                                                \end{pmatrix}
$ and $B=\begin{pmatrix}
        m&n&l\\p&q&r
       \end{pmatrix}
$. Then $A\odot B$ would be the $6\times 3$ matrix obtained as: $\begin{pmatrix}
                   \begin{pmatrix}
                   a\\d\\x
                   \end{pmatrix}\otimes\begin{pmatrix}m\\p\end{pmatrix}&\begin{pmatrix}
                   b\\e\\y
                   \end{pmatrix}\otimes\begin{pmatrix}n\\q\end{pmatrix}&\begin{pmatrix}
                   c\\f\\z
                   \end{pmatrix}\otimes\begin{pmatrix}l\\r\end{pmatrix}
                                                                 \end{pmatrix}
$, where we have the convention that $\begin{pmatrix}
                   a_1\\a_2\\a_3
                   \end{pmatrix}\otimes\begin{pmatrix}x_1\\x_2\end{pmatrix}=\begin{pmatrix}a_1x_1\\a_1x_2\\a_2x_1\\a_2x_2\\a_3x_1\\a_3x_2\end{pmatrix}$. 

\subsection{Concentration inequalities for Lipschitz functions of Gaussian random variables}\label{Subsection: Concentration inequalities for Lipschitz image of Gaussian random variables}
The following result is a very powerful concentration inequality, a nice proof of which could be found in the online lecture notes:\\ \url{https://www.stat.berkeley.edu/~mjwain/stat210b/Chap2_TailBounds_Jan22_2015.pdf} where it appears as Theorem 2.4.
\begin{thm}
Let $(X_1,X_2,\cdots X_n)$ be a vector of iid standard Gaussian variables, and let $f:\IR^n\longrightarrow\IR$ be $L$-Lipschitz with respect to the Euclidean norm. Then the variable\\ $f(X)-\IE(f(X))$ is sub-Gaussian with parameter at most $L$, and hence
\[
\IP\Big[|f(X)-\IE(f(X))|\geq t\Big]\leq 2e^{-\frac{t^2}{2L^2}}\quad(t\geq0).  
\]
\end{thm}

\subsection{Minimal eigenvalues of perturbed matrices}\label{Subsection: minimal eigenvalues of perturbed matrices}
\begin{lmma}\label{Lemma: close elements in L2 norm means the minimum eigenvalue is also close}
For a $n\times n$ real matrix $M$, let $\|M\|_2$ denotes the $L^2$-norm (i.e. $\|M\|_2:=\sup_{\|x\|\leq1}\|Mx\|$) and $\lambda_{\min}(M)$ denotes the least eigenvalue. For two positive semidefinite matrices $A,B$, we have that $\lambda_{\min}(A)\geq\lambda_{\min}(B)-\|A-B\|_2$.
\end{lmma}
\begin{proof}
Since $A$ and $B$ are positive semidefinite matrices, by spectral theorem it implies that $\min_{\|x\|=1}x^TAx=\lambda_{\min}(A)$ and $\min_{\|x\|=1}x^TBx=\lambda_{\min}(B)$. Now 
\[
-x^T(A-B)x\leq|x^T(A-B)x|\leq\|x\|\|(A-B)x\|\leq\|A-B\|_2, 
\]
so that we have 
\[
x^TAx=x^TBx+x^T(A-B)x\geq\lambda_{\min}(B)-\|A-B\|_2 
\]
for all $x\in\IR^n$, so that taking the minimum of the left hand side yields $\lambda_{\min}(A)\geq\lambda_{\min}(B)-\|A-B\|$, as desired.
\end{proof}

\begin{lmma}\label{Lemma: close elements in Frobenius normx means the minimum eigenvalue is also close}
Suppose $A$ and $B$ are two $n\times n$ positive semidefinite matrices such that for each $i,j=1,2,\cdots n$ we have $|A_{ij}-B_{ij}|\leq\frac{\epsilon}{n^2}$. Then $\lambda_{\min}(A)\geq\lambda_{\min}(B)-\epsilon$.
\end{lmma}

\begin{proof}
For a $n\times n$ matrix $M$, let $\|M\|_F$ denotes the Frobenius norm. Then it is a well known fact that $\|M\|_2\leq\|M\|_F$. Now
\begin{equation*}
\begin{split}
\|A-B\|_2&\leq\|A-B\|_F\\
&=\sqrt{\sum_{i,j=1}^n|A_{ij}-B_{ij}|^2}\\
&\leq\sum_{i,j=1}^n|A_{ij}-B_{ij}|\leq\epsilon.
\end{split}
\end{equation*}
Applying Lemma \ref{Lemma: close elements in L2 norm means the minimum eigenvalue is also close}, we have the result.
\end{proof}

\section{Dynamics of the Gramian matrix associated with the gradient of the loss function and related results}\label{Subsection: dynamics of the Gramian matrix associated with the gradient of the loss function} 
Setting $f(W,x,a):=\frac{1}{\sqrt{m}}\sum_{r=1}^m a_r\sigma(w_r^Tx)$, we define\\ $L(W,a):=\frac{1}{2}\sum_{i=1}^n\Big(y_i-f(W,x_i,a)\Big)^2$. $L(W,a)$ is a differentiable function with respect to the input weights $W$. We make a random assignment for $a_r$ so that for each $r=1,2,\cdots m$, $a_r\sim\text{unif}\{-1,1\}$. Let us explicitly compute $\nabla L(W,a)$. 

Recall each $w_r\in W$ is a vector in $\IR^d$, so that we can write $w_r=(w_r^{(1)},w_r^{(2)},\cdots w_r^{(d)})^T\in\IR^d$. This means that $\nabla L(W,a)$ is a function with variables in $\IR^{md}$. Now 
\[
\frac{\partial L(W,a)}{\partial w_r^{(l)}}=\sum_{i=1}^n(f(W,x_i,a)-y_i)\frac{\partial f(W,x_i,a)}{\partial w_r^{(l)}}. 
\]
We compute $\frac{\partial f(W,x_i,a)}{\partial w_r^{(l)}}$.
\begin{equation}
\begin{split}
&\frac{\partial f(W,x_i,a)}{\partial w_r^{(l)}}\\
&=\frac{1}{\sqrt{m}}\sum_{r=1}^ma_r\frac{\partial \sigma(w_r^Tx_i)}{\partial w_r^{(l)}}\\
&=\frac{1}{\sqrt{m}} a_r x_i^{(l)}\sigma^\prime(w_r^Tx_i).
\end{split}
\end{equation}
Let $e\in\IR^n$ be the vector whose components are given by $e_i:=f(W,x_i,a)-y_i$ for $i=1,2,\cdots n$. Let us define a $m\times n$ matrix $\cla$ whose entries are given by $\cla_{pq}:=\frac{1}{\sqrt{m}}a_p\sigma^\prime(w_p^Tx_q)$ for $p=1,2,\cdots m$ and $q=1,2,\cdots n$, as well as a $d\times n$ matrix $\mathcal{B}$ whose entries are given by $\mathcal{B}_{kl}:=x_l^{(k)}$ for $k=1,2,\cdots d$ and $l=1,2,\cdots n$. It then follows that 
\[
\frac{\partial L(W,a)}{\partial w_r^{(l)}}=\Big((\cla\odot\mathcal{B})e\Big)_{rl}\quad(r=1,2,\cdots m, l=1,2\cdots d), 
\]
or in other words we have that 
\[
\nabla L(W,a) = \Big((\cla\odot\mathcal{B})\Big)(e),
\]
where $\odot$ is the Khatri-Rao product of matrices as described in Subsection \ref{Subsection: Khatri-Rao product of matrices} and $\oplus$ denotes the direct sum of matrices. Let us look at a typical column of the $md\times n$ matrix $\cla\odot\mathcal{B}$. The j$^{th}$-column ($j=1,2,\cdots n$) looks like:
\[
\begin{pmatrix}
a_1\frac{1}{\sqrt{m}}\sigma^\prime(w_1^Tx_j)\\
a_2\frac{1}{\sqrt{m}}\sigma^\prime(w_2^Tx_j)\\
a_3\frac{1}{\sqrt{m}}\sigma^\prime(w_3^Tx_j)\\
\cdot\\
\cdot\\
\cdot\\
a_m\frac{1}{\sqrt{m}}\sigma^\prime(w_m^Tx_j)
\end{pmatrix}\otimes x_j.
\]
Let us consider the  matrix $(\cla\odot\mathcal{B})^T(\cla\odot\mathcal{B})$. The $(p,q)^{th}$ element of this matrix is given by 
\[
\Big(\sum_{r=1}^m \frac{1}{m}\sigma^\prime\Big(w_k^Tx_p\Big)\sigma^\prime\Big(w_k^Tx_q\Big)\Big)\ot x_p^Tx_q\quad(p,q=1,2,\cdots n),
\]
where we have used the fact that $a_r^2=1$ for all $r=1,2,\cdots m$.
Let us consider $m$ iid normal variates $Z:=\{z_i\}_{i=1}^m$, where $z_i\sim N(0,\mathbb{I}_d)$ for each $i=1,2,\cdots m$ and $z\sim N(0,\mathbb{I}_d)$. We will prove a result similar to \cite[Theorem 3.1]{gdescentOverPara1}. Consider the $n\times n$ matrix $H^\infty$ given by 
\[
H^\infty_{pq}:=\IE_{z\sim N(0,\mathbb{I}_d)}\Big[\sigma^\prime(z^Tx_p)\sigma^\prime(z^Tx_q)\Big]\ot x_p^Tx_q\quad(p,q=1,2,\cdots n).
\]

\begin{thm}\label{Theorem: Theorem 3.1 in overparametrized paper}
$\lambda_0:=\lambda_{\min}(H^\infty)>0$.    
\end{thm}
\begin{proof}
Let $\Omega$ denotes the measure space $\IR^d$ with the measure induced by the random variable $z$ and $L^2(\Omega,\IR^d)$ denotes the Hilbert space of square integrable $\IR^d$-valued functions, so that for $f,g\in L^2(\Omega,\IR^d)$, we have that \[\innerl f,g\innerr_{L^2(\Omega,\IR^d)}=\IE_{z\sim N(0,\mathbb{I}_d)}\Big[f(z,v)^Tg(z,v)\Big].\] For $x\in\IR^d$, let us define $\phi(x):\Omega\longrightarrow\IR^d$ by $\phi(x)(\omega):=\sigma^\prime(\omega^T x)\ot x$. Clearly, $\phi(x)\in L^2(\Omega,\IR^d)$ for all $x\in\IR^d$. To prove the hypothesis of the theorem, we may note that the matrix $H^\infty$ is the Gramian matrix given by $\Big(\Big(\innerl\phi(x_i),\phi(x_j)\innerr\Big)\Big)_{i,j=1}^n$, so that the problem boils down to proving that the vectors $\{\phi(x_i)\}_{i=1}^n$ are linearly independent in $L^2(\Omega,\IR^d)$. So let $\sum_{i=1}^n\alpha_i\phi(x_i)=0$ for some scalars $\{\alpha_i\}_{i=1}^n$. This means for almost all $\omega\in\Omega$, we have that $\sum_{i=1}^n\alpha_i\phi(x_i)(\omega)=0$. Now $\Omega$ is a topological space (carries the topology of $\IR^d$) and the measure on $\Omega$ induced by $z$ is a Radon measure, which assigns positive mass to open subsets of $\Omega$. $\Omega\ni\omega\mapsto\sum_{i=1}^n\alpha_i\phi(x_i)(\omega)\in\IR^d$ is a continuous $\IR^d$-valued function which is almost everywhere (with respect to this measure on $\Omega$) zero. Thus it must be zero everywhere, i.e. we have $\sum_{i=1}^n\alpha_i\phi(x_i)(\omega)=0$ for all $\omega\in\Omega$ i.e. $\sum_{i=1}^n\alpha_i\sigma^\prime(\omega^Tx_i)(x_i)=0$ for all $\omega\in\Omega$, so that in particular $\sum_{i=1}^n\alpha_i\sigma^\prime(\lambda\omega^Tx_i)(x_i)=0$ for all $\lambda\in\IR$ and $\omega\in\Omega$. Let us select $\omega$ such that for each $i=1,2,\cdots n$ the numbers $f_i:=\omega^Tx_i$ are different from each other i.e. $f_i\neq f_j$ for $i\neq j$ (See Section \ref{Appendix} for a proof of this). Also, let us assume without loss of generality that $|f_1|>|f_2|>|f_3|>\cdots |f_n|$. With this $\omega$ we have that 
\begin{equation}\label{Equation: lambda equals zero}
\sum_{i=1}^n\alpha_i\sigma^\prime(\lambda f_i)(x_i)=0\quad(\forall~\lambda\in\IR). 
\end{equation}
It follows by a direct computation that $\frac{d^k}{dx^k}\sigma^\prime(x)|_{_{x=0}}\neq0$ for all $k$. Differentiating equation \eqref{Equation: lambda equals zero} with respect to $\lambda$ $k$-times and evaluating at zero, we see that 
\[
\frac{d^k}{dx^k}\sigma^\prime(x)|_{_{x=0}}\Big[\alpha_1x_1+\sum_{i=2}^n\alpha_i\Big(\frac{f_i}{f_1}\Big)^kx_i\Big]=0. 
\]
Recalling that for all $k$ we must have $\frac{d^k}{dx^k}\sigma^\prime(x)|_{_{x=0}}\neq0$, the above equation implies that 
\[
\alpha_1x_1+\sum_{i=2}^n\alpha_i\Big(\frac{f_i}{f_1}\Big)^kx_i=0\quad(\text{for all $k$}). 
\]
Now letting $k\longrightarrow\infty$, noting that $\Big|\frac{f_i}{f_1}\Big|<1$ for all $i=1,2,\cdots n$ so that $\Big(\frac{f_i}{f_1}\Big)^k\longrightarrow0$, we have that $\alpha_1x_1=0$ which implies that $\alpha_1=0$ as $x_1\neq0$. Considering the above sum from $2$ onwards, we can prove similarly that $\alpha_2=0,\alpha_3=0,\cdots\alpha_n=0$, which proves the linear independence, which in turn implies the hypothesis of the theorem.
\end{proof}

\begin{thm}\label{Theorem: with probability 1-ndelta the Gramian matrix is positive definite}
Given any $0<\delta<1$ with $\delta<\frac{1}{n}$, for $m>\frac{64c_1^2c_2^2n^2\ln(\frac{2n}{\delta})}{\lambda_0^2}$ (see Subsection \ref{Subsection: our setup} for the definitions of $c_1,c_2$), with probability at least $1-n\delta$ (where the probability is the product probability on $\IR^{md}$ induced by $m$ iid random variables $\{z_i\}_{i=1}^m$), the  matrix $\mathcal{C}$ described by 
\[
\mathcal{C}_{pq}:=\Big(\sum_{r=1}^m \frac{1}{m}\sigma^\prime\Big(z_r^Tx_p\Big)\sigma^\prime\Big(z_r^Tx_q\Big)\Big)\ot x_p^Tx_q,
\]
satisfies $\lambda_{\min}(\mathcal{C})>\frac{3}{4}\lambda_0$.
\end{thm}

\begin{proof}
For each $p,q$, let us define a random variable \[X_{pq}:=\sum_{r=1}^m \frac{1}{m}\sigma^\prime\Big(z_r^Tx_p\Big)\sigma^\prime\Big(z_r^Tx_q\Big),\] and the quantity given by 
\[
\mathsf{X}_{pq}:=\IE_{z\sim N(0,\mathbb{I}_d)}(\sigma^\prime(z^Tx_p)\sigma^\prime(z^Tx_q)).
\]It follows that $\mathbb{E}[X_{pq}]=\mathsf{X}_{pq}$.  $X_{pq}$ can be thought of as a function in $md$ variables (a function of the input weights). Let us write $\nabla(X_{pq})$ in a convenient form. Consider the vectors \[P_x:=\begin{pmatrix}
                                                                                                                                                                                                                                                                              \sigma^{\prime\prime}(z_1^Tx_p)\sigma^\prime(z_1^Tx_q)\\
                                                                                                                                                                                                                                                                              \sigma^{\prime\prime}(z_2^Tx_p)\sigma^\prime(z_2^Tx_q)\\
                                                                                                                                                                                                                                                                              \cdots\\
                                                                                                                                                                                                                                                                              \sigma^{\prime\prime}(z_m^Tx_p)\sigma^\prime(z_m^Tx_q)
                                                                                                                                                                                                                                                                             \end{pmatrix}
\] and \[P_y:=\begin{pmatrix}
                                                                                                                                                                                                                                                                              \sigma^{\prime}(z_1^Tx_p)\sigma^{\prime\prime}(z_1^Tx_q)\\
                                                                                                                                                                                                                                                                              \sigma^{\prime}(z_2^Tx_p)\sigma^{\prime\prime}(z_2^Tx_q)\\
                                                                                                                                                                                                                                                                              \cdots\\
                                                                                                                                                                                                                                                                              \sigma^{\prime}(z_m^Tx_p)\sigma^{\prime\prime}(z_m^Tx_q)
                                                                                                                                                                                                                                                                             \end{pmatrix}.\] Then it follows that $\nabla(X_{pq})=\frac{1}{m}\big(P_x\otimes x_p+P_y\otimes x_q\Big)$. Recalling that $\|x_i\|\leq1$ for all $i=1,2,\cdots n$, it follows that 
\begin{equation}
\begin{split}
\|\nabla(X_{pq})\|&\leq\frac{1}{m}(\|P_x\ot x_p\|+\|P_y\ot x_q\|)\\
&\leq\frac{1}{m}(\|P_x\|\| x_p\|+\|P_y\|\| x_q\|)\quad(\text{using $\|P_x\ot x_p\|\leq\|P_x\|\|x_p\|$})\\
&\leq\frac{1}{m}(\|P_x\|+\|P_y\|).
\end{split}
\end{equation}                                                                                                                                                                                                                                                                             We have that $\sup_{x\in\IR}|\sigma^\prime(x)|< c_1$ and $\sup_{x\in\IR}|\sigma^{\prime\prime}(x)|<c_2$ by the assumption on activation function, as mentioned in Subsection \ref{Subsection: our setup}. This implies that \begin{equation}\begin{split}\|P_x\|&=\sqrt{\sum_{r=1}^m\Big(\sigma^{\prime}(z_r^Tx_p)\sigma^{\prime\prime}(z_r^Tx_q)\Big)^2}\\&\leq c_1c_2\sqrt{m}\end{split}\end{equation} and similarly $\|P_y\|\leq c_1c_2\sqrt{m}$. This implies that $\|\nabla(X_{pq})\|\leq\frac{2c_1c_2}{\sqrt{m}}$.                                                                                                                                                                                                                            By virtue of mean value theorem it now follows that $X_{pq}$ can be regarded as a Lipschitz function with parameter $\frac{2c_1c_2}{\sqrt{m}}$. Thus by the inequality stated in Subsection \ref{Subsection: Concentration inequalities for Lipschitz image of Gaussian random variables} we have that for any $s>0$
                                                                                                                                                                                                                                                                             \[
                                                                                                                                                                                                                                                                  \IP[|X_{pq}-\IE[X_{pq}]|\geq s]\leq2e^{-\frac{s^2m}{4c_1^2c_2^2}}. 
                                                                                                                                                                                                                                                                             \]
Putting $s=\sqrt{\ln(\frac{2n}{\delta})\frac{4c_1^2c_2^2}{m}}$, we get 
\[
\IP\Big[|X_{pq}-\IE[X_{pq}]|\geq \sqrt{\ln(\frac{2n}{\delta})\frac{4c_1^2c_2^2}{m}}\Big]\leq 2e^{-\ln(\frac{2n}{\delta})}=\frac{\delta}{n}. 
\]
We now have
\begin{equation*}
\begin{split}
&P\Big[\sum_{p,q=1}^n|X_{pq}-\mathsf{X}_{pq}|^2\geq n^2\ln(\frac{2n}{\delta})\frac{4c_1^2c_2^2}{m}\Big]\\
&=P\Big[\sum_{p,q=1}^n|X_{pq}-E[X_{pq}]|^2\geq n^2\ln(\frac{2n}{\delta})\frac{4c_1^2c_2^2}{m}\Big]\\
&\leq P\Big[\cup_{p,q=1}^n\Big(|X_{pq}-E[X_{pq}]|\geq\sqrt{\ln(\frac{2n}{\delta})\frac{4c_1^2c_2^2}{m}}\Big)\Big]\\
&\leq\sum_{p,q=1}^nP\Big[|X_{pq}-E[X_{pq}]|\geq\sqrt{\ln(\frac{2n}{\delta})\frac{4c_1^2c_2^2}{m}}\Big]\quad(\text{using Boole's inequality})\\
&\leq n\delta.
\end{split}
\end{equation*}
Let us relate the elements $X_{pq},\IE[X_{pq}]$ and matrices $\mathcal{C}$ and $H^\infty$. We may note that \[\mathcal{C}_{pq}=X_{pq}x_p^Tx_q\] and \[H^\infty_{pq}=\IE[X_{pq}]x_p^Tx_q.\]
Let us also recall that $\|x_i\|\leq1$ for all $i=1,2,\cdots n$. Thus we have that 
\begin{equation*}
\begin{split}
\|\mathcal{C}-H^\infty\|_2&\leq\|\mathcal{C}-H^\infty\|_F\\
&=\sqrt{\sum_{p,q=1}^n\Big|X_{pq}x_p^Tx_q-\IE[X_{pq}]x_p^Tx_q\Big|^2}\\
&=\sqrt{\sum_{p,q=1}^n\Big|x_p^Tx_q\Big|^2\Big|X_{pq}-\IE[X_{pq}]\Big|^2}\\
&\leq\sqrt{\sum_{p,q=1}^n\Big|X_{pq}-\IE[X_{pq}]\Big|^2}\quad(\text{as $\|x_i\|\leq1$ for all $i$}).
\end{split}
\end{equation*}
Thus we have that for any $t>0$, $P\Big[\|\mathcal{C}-H^\infty\|_2<t\Big]\geq P\Big[\sqrt{\sum_{p,q=1}^n\Big|X_{pq}-\IE[X_{pq}]\Big|^2}<t\Big]$.
This in turn means \[
P\Big[\|\mathcal{C}-H^\infty\|_2\leq\sqrt{n^2\ln(\frac{2n}{\delta})\frac{4c_1^2c_2^2}{m}}\Big]\geq 1-n\delta. 
\]
Since $n\delta<1$, this means that the event $\Big\{\|\mathcal{C}-H^\infty\|_2\leq\sqrt{n^2\ln(\frac{2n}{\delta})\frac{4c_1^2c_2^2}{m}}\Big\}$ has a probability of at least $1-n\delta$. Now $\sqrt{n^2\ln(\frac{2n}{\delta})\frac{4c_1^2c_2^2}{m}}<\frac{\lambda_0}{4}$ would imply that $m>\frac{64c_1^2c_2^2n^2\ln(\frac{2n}{\delta})}{\lambda_0^2}$, which means with probability at least $1-n\delta$, $\lambda_{\min}(\mathcal{C})>\frac{3\lambda_0}{4}$, as required.
\end{proof}

\subsection{Convergence of gradient descent}\label{Subsection: Convergence of the gradient descent}
In our discourse, we will not train the output weights, we will only train the input weights. Let us make a random initialization of the output weights $\{a_r\}_{r=1}^m$ by selecting them uniformly from the set $\{-1,1\}$. Then the loss function $L(W,a)$ becomes a function of the input weights only, so that we can regard it as a function in $md$ variables (recall that each of the input weights $w_i\in\IR^d$ and there are $m$ many of them). The gradient descent then proceeds by updating the input weights as:
\[
W(k+1):=W(k)-\eta\nabla L_t(W(k),a), 
\]
where $\eta>0$ is the step size and $W(0)$ is some random initialization and by $W(k)$ we are denoting the vector consisting of input weights i.e. $W(k):=(w_1(k),w_2(k),\cdots w_m(k))^T$ which is a vector in $\IR^{md}$. We would like to select a very small step size, and moreover, let us note that our loss function is differentiable as $\sigma$ is a differentiable function. So we can rephrase the above equation in terms of a differential equation:
\[
\frac{dw_r(t)}{dt}=-\frac{\partial L(W(t))}{\partial w_r(t)}, 
\]a solution of which will be a continuous curve in $\IR^{md}$ given by $t\mapsto W(t)$.
At any given time point $W(t):=\{w_r(t)\}_{r=1}^m$, let us consider the matrix $\mathcal{C}[W(t)]$ described by 
\[
\mathcal{C}[W(t)]_{pq}:=\Big(\sum_{r=1}^m \frac{1}{m}\sigma^\prime\Big(w_k(t)^Tx_p\Big)\sigma^\prime\Big(w_k(t)^Tx_q\Big)\Big)\ot x_p^Tx_q). 
\]
We now closely follow the technique outlined in \cite[p. 5--6]{gdescentOverPara1}. Let $u_i(t):=f(W(t),x_i)$ for $i=1,2,\cdots n$. Then $\frac{du_i(t)}{dt}=\sum_{r=1}^m\Big\langle\frac{\partial f(W(t),x_i)}{\partial w_r(t)},\frac{dw_r(t)}{dt}\Big\rangle=\sum_{j=1}^n\mathcal{C}[W(t)]_{ij}(y_j-u_j)$ for $i=1,2,\cdots n$, so that letting $u(t):=(u_1,u_2,\cdots u_n)^T$ and $y:=(y_1,y_2,\cdots y_n)^T$, we can write in the vectorial form as $\frac{du(t)}{dt}=\mathcal{C}[W(t)](y-u)$. The following lemma can be proven exactly as in \cite[Lemma 3.3]{gdescentOverPara1}, so we skip the proof.

\begin{lmma}\label{Lemma: Lemma 3.3 from Overparametrization}
 Suppose for $0\leq s\leq t$ we have that $\lambda_{\min}(\mathcal{C}[W(s)])>\frac{1}{2}\lambda_0$. Then 
 \begin{itemize}
  \item[(a)]
  $\|y-u(t)\|^2\leq e^{-\lambda_0 t}\|y-u(0)\|^2$.
  \item[(b)]
  For all $r=1,2,3,\cdots m$, we have $\|w_r(t)-w(0)\|\leq\frac{\sqrt{n}\|y-u(0)\|}{\sqrt{m}\lambda_0}$.
 \end{itemize}
\end{lmma}
\begin{rmrk}\label{Remark: Lipschitz coefficient allows us to control least eigenvalue in general}
Let us observe the elements of the matrix $\mathcal{C}[W(t)]$ more closely, namely elements of the form 
\[
\mathcal{C}[W(t)]_{pq}:=\Big(\sum_{r=1}^m \frac{1}{m}\sigma^\prime\Big(w_r(t)^Tx_p\Big)\sigma^\prime\Big(w_r(t)^Tx_q\Big)\Big)\ot x_p^Tx_q. 
\]
For a fixed $p,q$, thinking of $\mathcal{C}[W(t)]_{pq}$ as a function in $md$ variables, following the proof of Theorem \ref{Theorem: with probability 1-ndelta the Gramian matrix is positive definite} and recalling that $\|x_i\|\leq1$ for all $i=1,2,\cdots n$, we must have that $\|\nabla(\mathcal{C}[W(t)]_{pq})\|\leq\frac{4c_1c_2}{\sqrt{m}}$, so that by virtue of mean value theorem we have that 
\[
\Big|\mathcal{C}[W(t)]_{pq}-\mathcal{C}[W(s)]_{pq}\Big|\leq\frac{4c_1c_2}{\sqrt{m}}\|W(t)-W(s)\|, 
\]
where $W(v):=\begin{pmatrix}
w_1(v)\\w_2(v)\\\cdots\\\cdots\\\cdots\\w_m(v)              
             \end{pmatrix}
\in\IR^{md}$ \Big(recall for each $r$, $w_r(v)\in\IR^d$ so that\\ $\|W(t)-W(s)\|=\sqrt{\sum_{r=1}^m\|w_r(t)-w_r(s)\|^2}$\Big). This immediately implies that if it turns out that $\|W(t)-W(s)\|<\frac{\sqrt{m}\lambda_0}{16c_1c_2n}$, we have that $\|\mathcal{C}[W(t)]-\mathcal{C}[W(s)]\|_F<\frac{\lambda_0}{4}$. We will be using this in the sequel.
\end{rmrk}

It remains to be computed what is the probability of the event that for $0\leq s\leq t$, $\lambda_{\min}(\mathcal{C}[W(s)])>\frac{\lambda_0}{2}$. We summarize this in the next proposition.

\begin{ppsn}\label{Proposition: Main theorem}
Suppose the training data $x:=\{x_i\}_{i=1}^n$ and the corresponding responses $\{y_i\}_{i=1}^n$ satisfy:
\begin{itemize}
    \item 
$\|x_i\|\leq1$ for all $i=1,2,\cdots n$;
\item
$\|y_i\|<\kappa$ for some constant $\kappa$, for all $i=1,2,\cdots n$.
\end{itemize}
Let $C$ denotes the event that for all $s\geq0$, $\lambda_{\min}(\mathcal{C}[W(s)])>\frac{\lambda_0}{2}$. We have

\[\IP[C]\geq1-n\delta-\frac{K}{\ln(\frac{2n}{\delta})}\]
for some constant $K$ independent of $n$ and $\delta$.
\end{ppsn}
\begin{proof}
Note that $\|y-u(0)\|$ is a random variable defined on $\IR^{md}$ in its own right. Let us define the events $A_\gamma$ for $\gamma>0$ and $B$ as follows:
\[
A_\gamma:=\Big\{\|y-u(0)\|\leq\frac{1}{\gamma}\Big\}\quad(\gamma>0); 
\]
\[
B:=\Big\{\lambda_{\min}[\mathcal{C}[W(0)]>\frac{3\lambda_0}{4}\Big\}. 
\]
Let us note that $\{A_\gamma\}_\gamma$ is an increasing net of events with decreasing values of $\gamma$ i.e. $A_\gamma\subset A_{\gamma^\prime}$ if $\gamma^\prime<\gamma$. Moreover we have that $\cup_{\gamma}A_\gamma=\IR^{md}$. By continuity property of probability measures, it now follows that 
\[
\lim_{\gamma\longrightarrow0^+}\IP[B\cap A_\gamma]=\IP[B], 
\]
i.e. $\lim_{\gamma\longrightarrow0^+}\IP[B\cap A_\gamma]>1-n\delta$. 

Let us now observe closely the quantity $\|y-u(0)\|$. Let us recall that $u(0)=\begin{pmatrix}
u_1(0)\\u_2(0)\\u_3(0)\\\cdots\\\cdots\\u_n(0)                                                                                
                                                                               \end{pmatrix}
$, where $u_i(0):=\frac{1}{\sqrt{m}}\sum_{r=1}^ma_r\sigma(w_r(0)^Tx_i)$, and $y:=\begin{pmatrix}
                                                 y_1\\y_2\\y_3\\\cdots\\\cdots\\y_n                                   
                                                                                   \end{pmatrix}
$. Then 
\begin{equation*}
\begin{split}
\|y-u(0)\|^2&=\sum_{i=1}^n|y_i-u_i(0)|^2\\
&=\sum_{i=1}^n\Big[y_i^2+2y_iu_i(0)+u_i^2(0)\Big],
\end{split}
\end{equation*}
so that taking expectation with respect to all the initializing variables (recall $a_r\sim\text{unif}\{-1,1\}$, so that $\IE[a_r]=0$ and $a_r$'s are independent from $w_r(0)$'s) we have
\begin{equation*}
\begin{split}
&\IE[\|y-u(0)\|^2]\\
&=\sum_{i=1}^n\Big[y_i^2+\IE[u_i(0)^2]\Big]\\
&=\sum_{i=1}^n\Big[y_i^2+\frac{1}{m}\IE[\sum_{r=1}^m a_r^2\sigma(w_r(0)^Tx_i)^2+\sum_{r\neq r^\prime}a_ra_{r^\prime}\sigma(w_r(0)^Tx_i)\sigma(w_{r^\prime}(0)^Tx_i)]\Big]\\
&=\sum_{i=1}^n\Big[y_i^2+\frac{1}{m}\sum_{r=1}^m \IE[\sigma(w_r(0)^Tx_i)^2]\Big]\\
&=\sum_{i=1}^n\Big[y_i^2+\IE[\sigma(w^Tx_i)^2]\Big],
\end{split}
\end{equation*}
where $w\sim N(0,\mathbb{I}_d)$. Since $\IE[\sigma(w^Tx)^2]<c_3$ (the inequality follows from the assumption in Subsection \ref{Subsection: our setup}), and $|y_i|<\kappa$ for all $i=1,2,\cdots n$, for some number $\kappa$, we have 
\[
\IE[\|y-u(0)\|^2]\leq (\kappa^2+c_3)n.\]
Since $(\IE[\|y-u(0)\|])^2\leq\IE[\|y-u(0)\|^2]$, this implies that $\IE[\|y-u(0)\|]\leq D\sqrt{n}$, where $D:=\sqrt{\kappa^2+c_3}$. Thus by Markov's inequality, we have that $\IP[A_\gamma]\geq1-D\sqrt{n}\gamma$. Let $\gamma_0:=\frac{1}{4c_1c_2\sqrt{n}\ln(\frac{2n}{\delta})}$. Recalling that $m>\frac{64c_1^2c_2^2 n^2\ln(\frac{2n}{\delta})}{\lambda_0^2}$, we see that 
\[
\frac{\sqrt{n}\|y-u(0)\|}{\sqrt{m}\lambda_0}<\frac{\|y-u(0)\|}{8c_1c_2\sqrt{n\ln(\frac{2n}{\delta})}}, 
\]
so that occurrence of the event $A_{\gamma_0}$ implies \[\frac{\sqrt{n}\|y-u(0)\|}{\sqrt{m}\lambda_0}<\frac{\|y-u(0)\|}{8c_1c_2\sqrt{n\ln(\frac{2n}{\delta})}}\leq\frac{1}{8c_1c_2\sqrt{n\ln(\frac{2n}{\delta})}\gamma_0}=\frac{1}{2}\sqrt{\ln\Big(\frac{2n}{\delta}\Big)}<\frac{\sqrt{m}\lambda_0}{16c_1c_2n}.\]
We now prove that occurrence of the event $B\cap A_{\gamma_0}$ implies the occurrence of the event $C$. Suppose $B\cap A_{\gamma_0}$ has occurred. Suppose there exists some $t>0$ such that $\lambda_{\min}(\mathcal{C}[W(t)])\leq\frac{\lambda_0}{2}$. This would imply that $\|W(t)-W(0)\|\geq\frac{\sqrt{m}\lambda_0}{16c_1c_2n}$ as otherwise by Remark \ref{Remark: Lipschitz coefficient allows us to control least eigenvalue in general}, combined with Lemma \ref{Lemma: close elements in Frobenius normx means the minimum eigenvalue is also close}, it would imply that $\lambda_{\min}(\mathcal{C}[W(t))>\frac{\lambda_0}{2}$. Let \[t_0:=\inf\{t\in(0,+\infty):~\|W(t)-W(0)\|\geq\frac{\sqrt{m}\lambda_0}{16c_1c_2n}\}.\] Since $t\mapsto W(t)$ is a continuous curve (being a solution of a differential equation with differentiable coefficients), this implies that $t_0>0$ and we must have $\|W(t_0)-W(0)\|=\frac{\sqrt{m}\lambda_0}{16c_1c_2n}$. Now for $s\in(0,t_0)$ we must have that $\|W(s)-W(0)\|<\frac{\sqrt{m}\lambda_0}{16c_1c_2n}$ implying by Remark \ref{Remark: Lipschitz coefficient allows us to control least eigenvalue in general} and Lemma \ref{Lemma: close elements in Frobenius normx means the minimum eigenvalue is also close} that $\lambda_{\min}(\mathcal{C}[W(s)])>\frac{\lambda_0}{2}$. By Lemma \ref{Lemma: Lemma 3.3 from Overparametrization}-(b) we arrive at the inequality \[\|W(s)-W(0)\|\leq\frac{\sqrt{n}\|y-u(0)\|}{\sqrt{m}\lambda_0}<\frac{\sqrt{m}\lambda_0}{16c_1c_2n}\quad(\text{for all $s\in(0,t_0)$},\] which by virtue of continuity implies $\|W^\prime(t_0)-W^\prime(0)\|\leq\frac{\sqrt{n}\|y-u(0)\|}{\sqrt{m}\lambda_0}<\frac{\sqrt{m}\lambda_0}{16c_1c_2n}$, contradiction to the fact that $\|W^\prime(t_0)-W^\prime(0)\|=\frac{\sqrt{m}\lambda_0}{16c_1c_2n}$. This proves that $B\cap A_{\gamma_0}\implies C$ so that as $B\cap A_{\gamma_0}\subset C$, we have $\IP[C]\geq\IP[B\cap A_{\gamma_0}]$. Let us compute $\IP[B\cap A_{\gamma_0}]$ (in the following $\overline{A}$ denotes complement of the set $A$).

\begin{equation*}
\begin{split}
\IP[B\cap A_{\gamma_0}]&=1-\IP\Big[\overline{B\cap A_{\gamma_0}}\Big]\\
&=1-\IP\Big[\overline{B}\cup\overline{A_{\gamma_0}}\Big]\\
&\geq1-\IP[\overline{B}]-\IP[\overline{A_{\gamma_0}}]\\
&=\IP[B]+\IP[A_{\gamma_0}]-1\\
&\geq\IP[A_{\gamma_0}]-n\delta\\
&\geq1-D\sqrt{n}\gamma_0-n\delta\quad(\text{recall $D:=\sqrt{\kappa^2+c_3}$})\\
&=1-n\delta-\frac{D}{4c_1c_2\ln(\frac{2n}{\delta})},
\end{split}
\end{equation*}
so that we ended up proving that $\IP[C]\geq1-n\delta-\frac{D}{4c_1c_2\ln(\frac{2n}{\delta})}$. This finishes the proof.
\end{proof}

Thus we ended up proving the following result:

\begin{thm}\label{Theorem: main result}
Consider an activation function $\sigma(x)$ satisfying the assumptions mentioned in Subsection \ref{Subsection: our setup}. 
Suppose we are given a training data $\{x_i\}_{i=1}^n$ with $x_i\in\IR^d$, $\|x_i\|\leq1$, such that $x_i \neq x_j \; \forall i,j$, and responses $\{y_i\}_{i=1}^n$ with $y_i\in\IR$ and $|y_i|<\kappa$ for some number $\kappa$. Let $\lambda_0$ be the minimum eigenvalue of the matrix $H^\infty$, whose entries are given by 
\[
H^\infty_{pq}:=\IE_{z\sim N(0,\mathbb{I}_d)}\Big[\sigma^\prime(z^Tx_p)\sigma^\prime(z^Tx_q)\Big]\ot x_p^Tx_q\quad(p,q=1,2,\cdots n),.
\]
Then we have the following:

\begin{itemize}

\item[(a)]
$\lambda_0>0$.

\item[(b)]
Fix a $\delta>0$ such that $n\delta+\frac{D}{4c_1c_2\ln(\frac{2n}{\delta})}<1$, where $D:=\sqrt{\kappa^2+c_3}$. 
Let us select $m>\frac{64c_1^2c_2^2n^2\ln(\frac{2n}{\delta})}{\lambda_0^2}$ ($c_1,c_2$ are the constants appearing in the assumption in Subsection \ref{Subsection: our setup}) and consider the network with a single hidden layer:\[f(W,x,a):=\frac{1}{\sqrt{m}}\sum_{r=1}^m a_r\sigma(w_r^Tx).\]
Then with random initializations: $a_r\sim\text{unif}\{-1,1\}$ and $w_r(0)\sim N(0,\mathbb{I}_d)$ for all $r=1,2,\cdots m$ and the condition that we do not train the output layer i.e. $a_r$'s are kept fixed upon initialization, with probability at least $1-\delta n-\frac{D}{4c_1c_2\ln(\frac{2n}{\delta})}$, gradient descent with small enough step size converges to $y_i$'s at an exponential rate.

\item[(c)]
Define $\delta' := \delta n+\frac{D}{4c_1c_2\ln(\frac{2n}{\delta})}$.
There exists a constant $C$ such that $m > \frac{C n^2\ln(\frac{2n}{\delta'})}{\lambda_0^2}$ implies $m>\frac{64c_1^2c_2^2n^2\ln(\frac{2n}{\delta})}{\lambda_0^2}$.
Then, as a corollary of point (b) we have that with random initialization as in the previous point, with probability at least $1-\delta'$, gradient descent with small enough step size converges to $y_i$'s at an exponential rate.

\end{itemize}
\end{thm}

\begin{rmrk}
It is perhaps worthwhile to point out that if we introduce the input biases, essentially the same analysis holds with according changes in the constants appearing in the threshold value of $m$ in Theorem \ref{Theorem: main result}.
\end{rmrk}


\section{Appendix}\label{Appendix}
\begin{lmma}
Suppose $\{x_i\}_{i=1}^n$ is a set of vectors in $\IR^d$ such that $x_i\neq x_j$ for $i\neq j$. Then we can select $w\in\IR^T$ such that the numbers $\{w^Tx_i\}_{i=1}^n$ are all different.
\end{lmma}
\begin{proof}
Consider the vectors in $\IR^{^{\Comb{n}{2}\cdot d}}$ given by $((x_i-x_j))_{i<j}$. Consider the map $f$ given by $f:\IR^d\ni w\mapsto ((w^Tx_i-w^Tx_j))_{i<j}\in\IR^{^{\Comb{n}{2}}}$. Clearly this is a continuous map. Consider the set $S$ in $\IR^{^{\Comb{n}{2}}}$ given by $S:=\{y:=(y_1,y_2,\cdots y_{_{\Comb{n}{2}}})\in\IR^{^{\Comb{n}{2}}}:~y_i\neq0$ for each $i\}$. Clearly, this is an open subset of $\IR^{^{\Comb{n}{2}}}$, so that $f$ being a continuous function, we must have that $f^{-1}(S)$ is an open subset of $\IR^d$. If $f^{-1}(S)=\emptyset$, this implies that $f^{-1}(S^c)=\IR^d$. Now $f$ is an open map (i.e. if $U\subset\IR^d$ is open $\implies$ $f(U)\subset\IR^{^{\Comb{n}{2}}}$ is also open), so that $f(\IR^d)=S^c$ should also be an open subset of $\IR^{^{\Comb{n}{2}}}$, so that $S^c$ is both a closed and an open subset of $\IR^{^{\Comb{n}{2}}}$, a contradiction to the fact that as a topological space, $\IR^{^{\Comb{n}{2}}}$ is connected. Thus $f^{-1}(S)\neq\emptyset$. Now we simply take $w\in f^{-1}(S)$, which serves the purpose.
\end{proof}

\bibliographystyle{unsrt}
\bibliography{bibliography}

\end{document}